\documentclass{amsart}

\usepackage{url}

\usepackage{amscd}
\usepackage{amsfonts}
\usepackage{amssymb}
\usepackage{euscript}
\usepackage{amsmath}
\usepackage{amsthm}
\usepackage[matrix, arrow, curve]{xy}
\usepackage{mathrsfs}

\usepackage{tikz}
\usepackage{tikz-cd}
\usepackage{dsfont}

\newcommand{\ze}{\mathop{\zeta}}
\newcommand{\sq}[1]{\mathbf{#1}}
\newcommand{\z}[1]{\ze{(\sq{#1})}}
\newcommand{\sh}{\mathop\mathrm{{sh}}}
\newcommand{\st}{\mathop{\mathrm{st}}}
\newcommand{\M}[1]{\mathcal{M}_{0, #1}}
\newcommand{\Md}[1]{\mathcal{M}^\delta_{0, #1}}
\newcommand{\Ms}[1]{\overline{\mathcal{M}}_{0, #1}}
\newcommand{\Msh}{\mathop{\mathfrak{sh}}}
\newcommand{\Mst}{\mathop{\mathfrak{st}}}
\newcommand{\wt}[1]{\mathop{{w}}(\sq{#1})}
\newcommand{\set}[1]{\underline{\mathbf{#1}}}
\newcommand{\f}[2]{f_{#1}^{#2}}
\newcommand{\sw}[1]{\mathop{r_{#1}}}
\newcommand{\dlog}{\mathop{d\, log}}
\newcommand{\cyc}[1]{\mathop{\Delta}(#1)}

\theoremstyle{plain}
\newtheorem{prop}{Proposition}
\newtheorem{theorem}{Theorem}

\newtheorem*{conj}{Conjecture}
\theoremstyle{definition}

\theoremstyle{remark}

\begin{document}

\title{On generalized stuffle relations between cell-zeta values}

\begin{abstract}
We introduce a family of linear relations between 
 cell-zeta values
that have a form similar to
product map relations
and jointly with them imply stuffle relations between multiple zeta values.
\end{abstract}

\author{Nikita Markarian}

\email{nikita.markarian@gmail.com}

\date{}

\address{Institut des Hautes \'Etudes Scientifiques, Le Bois-Marie 35 rte de Chartres, 91440 Bures-sur-Yvette, France}

\maketitle

\section*{Introduction}

Multiple zeta values are special values of multiple zeta functions
on the one hand and special values of multiple polylogarithms on the other.
Hence, they may be represented either as sums of number series
or as integrals.

Multiple zeta values form an algebra over rational numbers. A product of two of them may be presented as a linear combination
of multiple zeta values with integer coefficients by means
of each representation. It gives two systems of relations which
multiple zeta values obey. 

The first one is called shuffle relations.
They immediately follow from  Fubini's theorem
applied to the integral representation.

The second family is called stuffle relations
and is given by rearrangement of summands in the product
of number series.
They are  not so evident
in the integral presentation. 
The equality  of  corresponding integrals is
based on  Fubini's theorem,  relation (\ref{relation}), 
which is a form of  Arnold's relation, and
some coordinate transformations.
These transformations are given by permutations of
coordinates in cubical coordinates, but in  simplicial coordinates
they are 
birational transformations.

These relations may be extended by 
regularizations of some equalities with divergent series, see
\cite{IKZ, Racinet}.
This extended system of relations is called regularized double shuffle relations.
The long-standing conjecture  states that
they imply 
all rational relations between multiple zeta
values.

Multiple zeta values of weight $n$ being values
of integrals are periods of the pair $(\Md{n+3}, \Md{n+3}\setminus\M{n+3})$
of a special kind  (\cite{GM}).
In \cite{B,BCS} all periods of such  pairs were studied.
In \cite{BCS} they were called cell-zeta values.

As well as multiple zeta values, 
cell-zeta values obey a lot of relations over rational numbers.
By the main result of \cite{B}, all cell-zeta values are rational 
combinations of multiple zeta values.
In light of this, it is natural to try to find
 a set  of relations on cell-zeta values,
which allows to express any cell-zeta value
in terms of multiple zeta values and  
 implies all known relation on multiple zeta values.
We suggest a candidate for this, which consists of  
two families of relations.
In \cite{BCS} another system of relations
on cell-zeta values is written down.
It would be very interesting to compare them.

The first family 
containing quadratic-linear relations was introduced in \cite{B}.  
It is analogous to shuffle relations and follows from  Fubini's theorem. 
In \cite{BCS} these relations called
product map relations. We suggest the term "generalized shuffle relations"
to emphasize
the similarity between our pair of families with the pair of shuffle and stuffle
relations. 

The second family is what we call generalized stuffle relations.
This is a family of linear relations following from the 
relative version of  Fubini's theorem.
They seem to be new.
These relations generalize above-mentioned 
manipulations with integrals in cubical coordinates.
It means in particular that  generalized shuffle relations 
and generalized stuffle relations 
imply usual stuffle relations.
This is the main result of the paper.

The natural question is whether  generalized shuffle and stuffle relations   imply
regularized double shuffle relations.
The answer seems to be no. But if we add relations implied by Stokes' theorem to the generalized shuffle and stuffle relations, the answer becomes positive.
In \cite{Kaneko} a set of linear relations between multiple zeta values is introduced. It is  proved there that these relations along with double shuffle relations imply regularized double shuffle relations.
Thus, to confirm a positive answer to the question above one needs to show that the Kaneko--Yamamoto 
integral-series relations follow from Stokes' theorem and generalized stuffle relations.
It is a subject of another paper \cite{KY}.

In contrast to  usual stuffle relations
(see nevertheless \cite{S}),
generalized stuffle relations have a  clear motivic nature.
A motivic version of these relations
is the subject of future research.

{\bf Acknowledgments.} I am grateful to S.~Charlton, A.~Khoroshkin, A.~Le\-vin and
E.~Panzer for fruitful discussions. 
I am very grateful to the referee for  helpful suggestions, which substantially improved the paper. I thank IHES for hospitality and excellent working environment.

\section{Shuffle and stuffle relations}

\subsection{Multiple zeta values}

Call a finite sequence of natural numbers $(k_1, \dots, k_n)$
convergent if $k_1\ge 2$. For a convergent sequence
$\sq{k}=(k_1,\dots ,k_n)$  
the multiple zeta value is defined by the integral (see e.~g. \cite{IKZ})
\begin{equation}
\z{k}=\int_{\Delta_{\wt{k}}} \omega_1(t_1)\wedge\dots\wedge \omega_{\wt{k}}(t_{\wt{k}}),
\label{iter}
\end{equation}
where $\Delta_{\wt{k}}=\{1>t_1>\dots>t_{\wt{k}}>0\}$, $\wt{k}=k_1+\dots +k_n$ is weight of the sequence and
$$
\omega_i(t)=
\begin{cases} dt/(1-t) & \mbox{if $i\in \{k_1, k_1+k_2, \dots, k_1+\dots+k_n \}$} \\
dt/t & \mbox{otherwise} \end{cases}
$$
Thus multiple zeta values are  iterated integrals.

A more conventional way to define multiple zeta values
is by series representation 
ascending to Euler, but we will not need it.

\subsection{Shuffle relations}

With a finite sequence of natural numbers 
$\sq{k}=(k_1,\dots ,k_n)$  associate the word 
$z_{\sq{k}}=x^{k_1-1}yx^{k_2-1}\dots x^{k_{n-1}}y$
of two letters $x$ and $y$. It establishes a bijection between
sequences and words ending in $y$.

A finite multiset is an unordered finite list with possible repetitions.
For a multiset $M$,  denote by $x\cdot M$
the result of the action of operation $x\cdot$ on $M$ elementwise. 

Define shuffle product $\sh(\,\cdot\,, \,\cdot\,)$ of two words  in letters $x$ and $y$
as a multiset of words given  by the recursive rule
\begin{equation}
\sh(v\cdot z_1, u\cdot z_2)= 
v\cdot \sh(z_1, u\cdot z_2)\mathbin{\cup}
u\cdot \sh(v\cdot z_1,  z_2),
\label{sh}
\end{equation} 
where $u, v\in\{x,y\}$ and  $\sh(1, z)=\sh(z, 1)= \{z\}$.

One may see that the shuffle product of two words ending in $y$
consists of words ending in $y$. It defines the shuffle product of sequences
of natural numbers, which we denote likewise by $\sh(\,\cdot\,, \,\cdot\,)$.
\begin{prop}[Shuffle relations]
Let $\sq{k}$ and $\sq{l}$ be convergent sequences.
Then
\begin{equation}
\z{k}\z{l}=\sum_{\sq{s} \in \sh(\sq{k}, \sq{l})} \z{s} 
\label{esh}
\end{equation}
\label{Sh}
\end{prop}
\begin{proof}
The statement is an immediate consequence of  Fubini's theorem.
It also follows from Proposition \ref{Gsh} below.  
\end{proof}

\subsection{Stuffle relations}

For a finite sequence of natural numbers 
$\sq{k}=(k_1,\dots ,k_{n-1})$,  denote 
the sequence $(k_1,\dots ,k_{n-1}, k_n)$ by $\sq{k}\cdot k_{n}$. Introduce the empty sequence $()$ such that
$()\cdot k=(k)$. 

Define stuffle product $\st(\,\cdot\,, \,\cdot\,)$ of two   sequences 
as a multiset  of
sequences given by the recursive rule
\begin{equation}
\st( \sq{k}\cdot x, \sq{l}\cdot y)= 
 \st( \sq{k}\cdot x, \sq{l})\cdot y\mathbin{\cup}
\st(\sq{k}, \sq{l}\cdot y)\cdot x\mathbin{\cup}
\st(\sq{k}, \sq{l})\cdot(x+y)
\label{st}
\end{equation} 
and by $\st((), \sq{k})=\st(\sq{k}, ())= \{\sq{k}\}$.
Note that 
$\st(\sq{k}, \sq{l})=\st(\sq{l}, \sq{k})$.

Stuffle relations follow easily from the series representation
of multiple zeta values. The proof of them in terms of integrals
may be found in \cite{G, B, S}.
We present it in a form convenient for our purposes.
\begin{prop}[Stuffle relations]
Let $\sq{k}$ and $\sq{l}$ be convergent sequences.
Then
\begin{equation}
\z{k}\z{l}=\sum_{\sq{s} \in \st(\sq{k}, \sq{l})} \z{s} 
\label{est}
\end{equation}
\label{St}
\end{prop}
\begin{proof}
Define cubical coordinates
on the standard simplex $\Delta_{k}=\{1>t_1>\dots>t_k>0\}$   by
$$
x_1=t_1 \quad x_2=t_2/t_1\quad \dots \quad x_k=t_k/t_{k-1}.
$$
Introduce notations:
\begin{equation*}
\f{a}{b}=\frac{\prod_{i=a}^{b}x_i}{1-\prod_{i=a}^{b}x_i}.
\end{equation*}
In cubical coordinates we can rewrite definition (\ref{iter})  as
\begin{equation}
\z{k}=\int_{\square}f_{1}^{k_1} \f{1}{k_1+k_2}\cdots \f{1}{\wt{k}} \, dV,
\label{zcube}
\end{equation}
where $dV= dx_1/x_1\wedge dx_2/x_2\dots$ is the standard volume form on the torus and symbol $\square$ here and below in this proof
means the unit cube $\{0<x_i<1\}$.

Fubini's theorem gives

\begin{multline}
\z{k}\z{l}=\\
\int_{\square} \f{1}{k_1}\f{1}{k_1+k_2}\cdots \f{1}{\wt{k}}\times \\
\f{\wt{k}+1}{\wt{k}+l_1}\f{\wt{k}+1}{\wt{k}+l_1+l_2}\cdots\f{\wt{k}+1}{\wt{k}+\wt{l}} \, dV
\label{e1}
\end{multline}

Thus we need to prove that the right hand side of 
(\ref{e1}) equals to the right hand side of (\ref{est}).

Introduce  following transformations of the cube
\begin{equation}
\sw{a}(x_i)=
\begin{cases}
x_{a+1-i} & \mbox{for  $i\le a$,}\\
x_i & \mbox{for $i>a$.}
\end{cases}
\label{trans}
\end{equation}


Applying $\sw{\wt{k}}$ to the right hand side of (\ref{e1}) we get
\begin{multline}
\mbox{r.h.s. (\ref{e1})}=\\
\int_{\square} \f{1+\wt{k}-k_1}{\wt{k}}\f{1+\wt{k}-k_1-k_2}{\wt{k}}\cdots \f{1}{\wt{k}}\times \\
\f{\wt{k}+1}{\wt{k}+l_1}\f{\wt{k}+1}{\wt{k}+l_1+l_2}\cdots\f{\wt{k}+1}{\wt{k}+\wt{l}} \, dV
\label{e2}
\end{multline}

To show that right hand sides of (\ref{e2}) and (\ref{est})
are equal, we will prove a more general 
equality
\begin{multline}
\sum_{\sq{s} \in \st(\sq{k}, \sq{l})} \ze(\cdots(\sq{s}\cdot m_1)\cdot m_2)\cdots)\cdot m_i)=\\
\int_{\square} \f{1+\wt{k}-k_1}{\wt{k}}\f{1+\wt{k}-k_1-k_2}{\wt{k}}\cdots \f{1}{\wt{k}}\times \\
\f{\wt{k}+1}{\wt{k}+l_1}\f{\wt{k}+1}{\wt{k}+l_1+l_2}\cdots\f{\wt{k}+1}{\wt{k}+\wt{l}} \times\\
\f{1}{\wt{k}+\wt{l}+m_1}\f{1}{\wt{k}+\wt{l}+m_1+m_2}\cdots \f{1}{\wt{k}+\wt{l}+m_1+\dots+m_i}\, dV
\label{e3}
\end{multline}
by induction on $\wt{k}+\wt{l}$.

The base of induction is given by (\ref{zcube}).

If $\sq{k}=()$,
there is nothing to prove.
If $\sq{l}=()$, the application of $\sw{\wt{k}}$
to the integral proves the  equality.

If both sequences are not empty, let $x=k_{\wt{k}}$ and $y=l_{\wt{l}}$
be the last terms of $\sq{k}$ and $\sq{l}$.
Introduce notations
$\sq{k}=\sq{k'}\cdot x $
and $\sq{l}=\sq{l'}\cdot y$.

Substituting in $\alpha=\prod_{i=1}^{a}x_i$  and $\beta=\prod_{i=a+1}^{b} x_i$ into the relation
\begin{equation}
\frac{1}{(1-\alpha)(1-\beta)}=\frac{\alpha}{(1-\alpha)(1-\alpha\beta)}+\frac{\beta}{(1-\beta)(1-\alpha\beta)}+\frac{1}{1-\alpha\beta}
\label{relation}
\end{equation}
we have for $b>a>1$
\begin{equation*}
\f{1}{a}\f{a+1}{b}=\f{1}{a}\f{1}{b}+ \f{a+1}{b}\f{1}{b}+ \f{1}{b}
\end{equation*}

Applying this to the product of the last factors in the second and the third lines
of (\ref{e3}) we get a sum of three terms. 

The first term is 
\begin{multline}
\int_{\square} \f{1+\wt{k}-k_1}{\wt{k}}\f{1+\wt{k}-k_1-k_2}{\wt{k}}\cdots \f{1}{\wt{k}}\times \\
\f{\wt{k}+1}{\wt{k}+l_1}\f{\wt{k}+1}{\wt{k}+l_1+l_2}\cdots\f{\wt{k}+1}{\wt{k}+\wt{l'}} \times\\
\f{1}{\wt{k}+\wt{l'}+y}\f{1}{\wt{k}+\wt{l'}+y+m_1}\cdots \f{1}{\wt{k}+\wt{l'}+y+m_1+\dots+m_i}\, dV
\label{e4}
\end{multline}
By the induction assumption it is equal to 
$\sum_{\sq{s} \in \st(\sq{k}, \sq{l'})} \ze(\cdots(\sq{s}\cdot y)\cdot m_1)\cdots)\cdot m_i)$.

The second term is 
\begin{multline}
\int_{\square} \f{1+\wt{k}-k_1}{\wt{k}}\f{1+\wt{k}-k_1-k_2}{\wt{k}}\cdots \f{1+x}{\wt{k}}\times \\
\f{\wt{k}+1}{\wt{k}+l_1}\f{\wt{k}+1}{\wt{k}+l_1+l_2}\cdots\f{\wt{k}+1}{\wt{k}+\wt{l}} \times\\
\f{1}{\wt{l}+\wt{k'}+x}\f{1}{\wt{l}+\wt{k'}+x+m_1}\cdots \f{1}{\wt{l}+\wt{k'}+x+\dots+m_i}\, dV
\label{e5}
\end{multline}
Applying  $\sw{\wt{k}+\wt{l}}$  to it we get
\begin{multline}
\mbox{(\ref{e5})} =
\int_{\square} \f{1+\wt{l}-l_1}{\wt{l}}\f{1+\wt{l}-l_1-l_2}{\wt{l}}\cdots \f{1}{\wt{l}}\times \\
\f{\wt{l}+1}{\wt{l}+k_1}\f{\wt{l}+1}{\wt{l}+k_1+k_2}\cdots\f{\wt{l}+1}{\wt{l}+\wt{k'}} \times\\
\f{1}{\wt{l}+\wt{k'}+x}\f{1}{\wt{l}+\wt{k'}+x+m_1}\cdots \f{1}{\wt{l}+\wt{k'}+x+\dots+m_i}\, dV
\label{e6}
\end{multline}
By the induction assumption it is equal to 
$\sum_{\sq{s} \in \st(\sq{l}, \sq{k'})} \ze(\cdots(\sq{s}\cdot x)\cdot m_1)\cdots)\cdot m_i)$.

The third term is 
\begin{multline}
\int_{\square} \f{1+\wt{k}-k_1}{\wt{k}}\f{1+\wt{k}-k_1-k_2}{\wt{k}}\cdots \f{1+x}{\wt{k}}\times \\
\f{\wt{k}+1}{\wt{k}+l_1}\f{\wt{k}+1}{\wt{k}+l_1+l_2}\cdots\f{\wt{k}+1}{\wt{k}+\wt{l'}} \times\\
\f{1}{\wt{k'}+\wt{l'}+x+y}\f{1}{\wt{k'}+\wt{l'}+x+y+m_2}\cdots \f{1}{\wt{k'}+\wt{l'}+x+y+\dots+m_i}\, dV
\label{e7}
\end{multline}
Applying  $\sw{\wt{k}+\wt{l'}}$  to it we get
\begin{multline}
\mbox{(\ref{e7})} =
\int_{\square} \f{1+\wt{l'}-l_1}{\wt{l'}}\f{1+\wt{l'}-l_1-l_2}{\wt{l'}}\cdots \f{1}{\wt{l'}}\times \\
\f{\wt{l'}+1}{\wt{l'}+k_1}\f{\wt{l'}+1}{\wt{l'}+k_1+k_2}\cdots\f{\wt{l'}+1}{\wt{l'}+\wt{k'}} \times\\
\f{1}{\wt{l'}+\wt{k'}+x+y}\f{1}{\wt{l'}+\wt{k'}+x+y+m_1}\cdots \f{1}{\wt{l'}+\wt{k'}+x+y+\dots+m_i}\, dV
\label{e8}
\end{multline}
By the induction assumption it is equal to 
$\sum_{\sq{s} \in \st(\sq{l'}, \sq{k'})} \ze(\cdots(\sq{s}\cdot (x+y))\cdot m_1)\cdots)\cdot m_i)$.

By (\ref{st}) the sum of these three terms
equals to 
$\sum_{\sq{s} \in \st(\sq{k}, \sq{l})} \ze(\cdots(\sq{s}\cdot m_1)\cdot m_2)\cdots)\cdot m_i)$,
which proves (\ref{e3}) and thus the proposition.
\end{proof}

\section{Generalized shuffle and stuffle relations}

\subsection{$\M{S}$}
For a finite set $S$, $|S|> 2$
denote by $\M{S}$ the moduli space of its embeddings
$S\hookrightarrow \mathbb{P}^1$
to the complex projective line 
considered up to the
action of the M\"obius group.
This is a smooth affine variety and it has a smooth
projective compactification $\Ms{S}$, which is the
moduli space of  stable curves.
The complement $\Ms{S}\setminus\M{S}$
is the union of normal crossing divisors.
These divisors are numerated by
  partitions of $S$ in two subsets with cardinalities more than $2$.


Present $S$ as a union of the three-element set $\{\underline{0}, \underline{\infty}, \underline{1}\}$ and a set with $n$ elements, where and below $|S|=n+3$.
For a point of $\M{S}$ introduce on $\mathbb{P}^1$
the coordinate such that coordinates of points labeled 
by  $\underline{0}, \underline{\infty}, \underline{1}$ are $0$, $\infty$ and $1$
correspondingly.
Coordinates of  points labeled by the finite set with $n$ elements
are called simplicial coordinates on $\M{S}$.
Thus a point of $\M{S}$ with simplicial coordinates $(t_1, \dots, t_n)$
is  $(0, \infty, 1, t_1, \dots t_n)$.
  
The algebra of regular differential
forms on $\M{S}$ with logarithmic singularities at
infinity
is generated by $1$-forms 
\begin{equation}
\omega_{ij}=\dlog (t_i-t_j) \quad 0\le i< j\le n+1 \quad ij\neq 0n+1,
\label{omegas}
\end{equation}
where $t_i$ are simplicial coordinates, $t_0=1$ and $t_{n+1}=0$.
The only relations between  them 
are Arnold's relations:
\begin{equation*}
\omega_{ij}\wedge\omega_{jk}+\omega_{jk}\wedge\omega_{ik}+\omega_{ik}\wedge\omega_{ij}=0
\end{equation*}

Rewrite these relations in new coordinates $\alpha=t_j/t_i$, $\beta=t_k/t_j$ and $t=t_i$:
\begin{multline}
\left(-\frac{d\alpha}{1-\alpha}+\frac{dt}{t}\right)\wedge\left(-\frac{d\beta}{1-\beta}+\frac{d\alpha}{\alpha}+\frac{dt}{t}\right)+\\
\left(-\frac{d\beta}{1-\beta}+\frac{d\alpha}{\alpha}+\frac{dt}{t}\right)\wedge\left(-\frac{d(\alpha\beta)}{1-\alpha\beta}+\frac{dt}{t}\right)+\\
\left(-\frac{d(\alpha\beta)}{1-\alpha\beta}+\frac{dt}{t}\right)\wedge \left(-\frac{d\alpha}{1-\alpha}+\frac{dt}{t}\right)=0
\label{calc}
\end{multline}
Collecting terms with $d\alpha \wedge d\beta$
we get relation (\ref{relation}), which is crucial in the proof of Proposition \ref{St}. 

\subsection{$\Md{S}$}

Let $S$ be a cyclically ordered set.
In \cite{B} the space $\Md{S}$ is introduced
and described in great detail.
It may be thought as a partial compactification of
$\M{S}$
$$
\M{S}\subset \Md{S}\subset  \Ms{S},
$$
which contains only those compactification divisors 
for which corresponding partitions of $S$ respect the cyclic order.
It is proved in \cite{B} that $\Md{S}$ is a smooth affine variety.

For any subset $T\subset S$ the forgetful map
$$
\Md{S}\to \Md{T}
$$
is defined, where the cyclic order on $T$ is the restriction of the one
from $S$. 
Being restricted on $\M{S}$ this map forgets
points labeled by elements of $S\setminus T$. 

As above present $S$ as a union of the three-element set $\{\underline{0}, \underline{\infty}, \underline{1}\}$ and a set with $n$ elements, and introduce a  cyclic order on it as follows:
$[\underline{0}, \underline{\infty}, \underline{1}, 1, \dots, n]$.
Consider the standard simplex
$\Delta_{n}=\{1>t_1>\dots>t_{n}>0\}$
in $\M{S}$, where $t_i$ are simplicial coordinates.
This set depends on the order of the labeling
set, but one may see that it depends only on the cyclic order.
Then, the closure of the standard simplex in $\Md{S}$
is compact; this is the Stasheff polytope.
Denote this subset of $\Md{S}$ by 
$\cyc{S}$ to emphasize its dependence
on the cyclic order.

Since $\cyc{S}$ is compact, a regular differential
form of top degree may be integrated by it.
For  a  regular differential
form with logarithmic singularities at infinity $\omega$ on $\Md{S}$ 
the integral $\int_{\cyc{S}} \omega$
is  a period of the pair $(\Md{S},\, \Md{S}\setminus\M{S})$, see
\cite{GM}.
In \cite{BCS} these numbers are called cell-zeta values.
By the very definition (\ref{iter}), 
multiple zeta values are examples of such numbers, convergence 
of  sequence $\sq{k}$ implies that the form 
has no poles on divisors from  $\Md{S}\setminus\M{S}$.

\subsection{Generalized shuffle relations}

Let  $\set{3}$
be a cyclically ordered
 set with three elements.
Define a 3-pointed cyclically ordered set  $\mathcal{T}$
as a pair of a  cyclically ordered set  $T$ and a
monotonic embedding
$\imath\colon\set{3}\hookrightarrow T$.

Let $\mathcal{T}_{1,2}$ be a pair
of 3-pointed cyclically ordered sets and 
$\imath_{1,2}\colon\set{3}\hookrightarrow T_{1,2}$
are corresponding embeddings. 
Let $T_1\coprod_{\set{3}}T_2$ be the colimit of the diagram in the  category of sets
given by these embeddings.
Denote by $\Msh(\mathcal{T}_1, \mathcal{T}_2)$  the set of  cyclically  ordered
sets 
given by all cyclic orders on   $T_1\coprod_{\set{3}}T_2$ for which
projections on $T_1$
and $T_2$ are monotonic.

For any $C\in \Msh(\mathcal{T}_1, \mathcal{T}_2)$
consider the map
\begin{equation}
\beta_{C}\colon \Md{C} \to \Md{T_1}\times \Md{T_2},
\end{equation}
which is the forgetful map on each factor.
In  \cite[2.7]{B} this map is called the product map.

The following proposition is taken from \cite{B, BCS},
where it is called product map relations.
 
\begin{prop}[Generalized shuffle relations]
Using notations as above
let $\phi$ and $\psi$ be regular top-degree differential
forms on $\Md{T_1}$ and $\Md{T_2}$ correspondingly.
Then
\begin{equation}
\left(\int_{\cyc{T_1}} \phi\right)\cdot
\left(\int_{\cyc{T_2}} \psi\right)=
\sum_{C\in \Msh(\mathcal{T}_1, \mathcal{T}_2)}\int_{\cyc{C}}\beta_{C}^*( \phi\boxtimes \psi)
\label{gsh}
\end{equation}
\label{Gsh}
\end{prop}
\begin{proof}
By  Fubini's theorem 
and because $\beta$ is an embedding containing the domain of integration,
the left hand side of (\ref{gsh})
equals to the integral of $\beta_{C}^*( \phi\boxtimes \psi)$
by  $\beta^{-1}(\cyc{T_1}\times\cyc{T_2})$.
The decomposition of the latter set in simplices corresponding 
to elements of  $\Msh(\mathcal{T}_1, \mathcal{T}_2)$ proves the statement.
For more details see \cite[Corollary 7.10]{B}.
\end{proof}

\subsection{Generalized stuffle relations}

Let  $\set{4}$
be a cyclically ordered
 set with four elements.
Define a 4-pointed cyclically ordered set  $\mathcal{T}$
as a pair of a  cyclically ordered set  $T$ and a
monotonic embedding
$\imath\colon\set{4}\hookrightarrow T$.

The Klein four-group $V$ acts on $\set{4}$.
Half of this group respects the cyclic order
and the other half reverses it.
For   $\nu\in V$ and a 4-pointed cyclically ordered set  $\mathcal{T}=(T, \imath)$
denote by $\mathcal{T}^\nu$ the 4-pointed cyclically ordered set 
 with the embedding
equal to $\imath$ composed with $\nu$
and with the cyclic ordered set equal to $T$
or to $T^{op}$ 
depending on
whether $\nu$ respects cyclic order on $\set{4}$ or not, where ${\,\cdot\,}^{op}$
means the same set with the opposite order. Denote the latter 
cyclically ordered set by $T^{\nu}$.

Let $\mathcal{T}_{1,2}$  be a pair of 4-pointed cyclically ordered sets
and 
$\imath_{1,2}\colon\set{4}\hookrightarrow T_{1,2}$ are corresponding embeddings.
Let $T_1\coprod_{\set{4}}T_2$ be the colimit of the diagram in the category of sets
given by these embeddings.
Denote by $\Mst(\mathcal{T}_1, \mathcal{T}_2)$ the set of cyclically  ordered
sets 
given by all cyclic orders on  
  $T_1\coprod_{\set{4}}T_2$
 for which 
projections on $T_1$
and $T_2$ are monotonic.

For any  $\nu\in V$, $C\in \Mst(\mathcal{T}_1, \mathcal{T}_2)$ and
$C_\nu\in \Mst(\mathcal{T}_1, \mathcal{T}_2^\nu)$
consider  maps
\begin{equation}
\begin{tikzcd}[row sep=tiny]
\gamma_{C}\colon \Md{C}\arrow[rd]& \\
 & \Md{T_1}\times \Md{T_2}\\
\gamma_{C_\nu}\colon \Md{C_\nu}\arrow[ru]& 
\end{tikzcd}
\end{equation}
which are  forgetful map on each factor.

\begin{prop}[Generalized stuffle relations]
Using notations as above
let $\nu$ be a non-trivial element of the Klein four-group and $\phi$ and $\psi$ be regular differential
forms on $\Md{T_1}$ and $\Md{T_2}$ correspondingly such that
$$\deg \phi+\deg \psi=|T_1|+|T_2|-7$$
Then
\begin{equation}
\sum_{C\in \Mst(\mathcal{T}_1, \mathcal{T}_2)}\int_{\cyc{C}}\gamma_{C}^*( \phi\boxtimes \psi)=
\epsilon\cdot\sum_{C_\nu\in \Mst(\mathcal{T}_1, \mathcal{T}_2^\nu)}\int_{\cyc{C_\nu}}\gamma_{C_\nu}^* (\phi\boxtimes \psi),
\label{gst}
\end{equation}
where $\epsilon=(-1)^{\small{\frac{(|T_2|-3)(|T_2|-2)}{2}}}$ if $\nu$ reverses the cyclic order on $\set{4}$
and $\epsilon=1$ if not.
\label{Gst}
\end{prop}
\begin{proof}
Note that there are only two possibilities: $\phi$ is a top-degree
form and degree of $\psi$ is one less
and vice versa.

Consider the forgetful projection $\Md{C}\to \Md{\set{4}}$.
By  Fubini's theorem,
\begin{equation}
\int_{\cyc{C}}\gamma_{C}^*( \phi\boxtimes \psi)=\int_{\cyc{\set{4}}}\left( \int_{\cyc{C}/\cyc{\set{4}}}\gamma_{C}^*( \phi\boxtimes \psi)\right),
\label{fiber}
\end{equation}
where $\int_{\cyc{C}/\cyc{\set{4}}}$ is the fiber-wise
integral of the projection. By the relative analog
of Proposition \ref{Gsh}, we have
\begin{equation*}
\sum_{C\in \Mst(\mathcal{T}_1, \mathcal{T}_2)}\int_{\cyc{C}/\cyc{\set{4}}}\gamma_{C}^*( \phi\boxtimes \psi)=
\left(\int_{\cyc{T_1}/\cyc{\set{4}}} \phi\right)\cdot
\left(\int_{\cyc{T_2}/\cyc{\set{4}}} \psi\right)
\end{equation*}
Because the cross-ratio of four points is invariant under
the Klein four-group, action of the Klein four-group on four points
of projective line may be lifted to the whole projective line.
It follows the equality of forms
on $\Md{\set{4}}$
\begin{equation*}
\int_{\cyc{T_2}/\cyc{\set{4}}} \psi=\epsilon\cdot\int_{\cyc{T_2^\nu}/\cyc{\set{4}}} \psi
\end{equation*}

Thus we get
\begin{equation*}
\begin{split}
\sum_{C\in \Mst(\mathcal{T}_1, \mathcal{T}_2)}\int_{\cyc{C}/\cyc{\set{4}}}\gamma_{C}^*( \phi\boxtimes \psi)=
\left(\int_{\cyc{T_1}/\cyc{\set{4}}} \phi\right)\cdot
\left(\int_{\cyc{T_2}/\cyc{\set{4}}} \psi\right)=\\
\epsilon\cdot\left(\int_{\cyc{T_1}/\cyc{\set{4}}} \phi\right)\cdot
\left(\int_{\cyc{T_2^\nu}/\cyc{\set{4}}} \psi\right)=
\epsilon\cdot\sum_{C_\nu\in \Mst(\mathcal{T}_1, \mathcal{T}_2^\nu)}\int_{\cyc{C_\nu}/\cyc{\set{4}}}\gamma_{C_\nu}^*( \phi\boxtimes \psi)
\end{split}
\end{equation*}
Integrating both sides by $\cyc{\set{4}}$ and using (\ref{fiber}) we get a 
proof of the proposition.
\end{proof}


\begin{theorem}
Generalized shuffle relations (\ref{gsh}) and 
generalized stuffle relations (\ref{gst})
jointly imply shuffle relations (\ref{esh}) and stuffle relations 
(\ref{est}).
\end{theorem}

\begin{proof}
From (\ref{omegas}) we see that for a sequence $\sq{k}$ the differential form
under the integral sign in (\ref{iter}) is a top-degree
form on $\M{\wt{k}+3}$. If  sequence $\sq{k}$ is convergent, then
the form comes from $\Md{\wt{k}+3}$ for the standard
cyclic order $[0, \infty, 1, t_1, \dots, t_{\wt{k}}]$ on the labeling set.
Given a pair of such forms corresponding to sequences $\sq{k}$ and $\sq{l}$,
applying Proposition \ref{Gsh} to them
with 
$$
T_1=[0, \infty, 1, t_1, \dots, t_{\wt{k}}] \quad\mbox{and} \quad T_2=[0, \infty, 1, s_1, \dots, s_{\wt{l}}]
$$
and the common subset $[0, \infty, 1]$
we get (\ref{sh}). Thus generalized shuffle relations
imply shuffle relations.

Now show that generalized shuffle and generalized stuffle relations jointly
imply stuffle relations. To do it we need to verify that
all steps of the proof of Proposition \ref{St} 
follow from generalized shuffle and generalized stuffle relations.

As it is mentioned in \cite{B,S}, 
the first relation (\ref{e1}) of the proof follows from the Proposition \ref{Gsh}
for differential forms corresponding to sequences 
$\sq{k}$ and $\sq{l}$ and for
$$
T_1=[0, \infty, 1, t_1, \dots, t_{\wt{k}}] \quad\mbox{and} \quad T_2=[0, \infty, t_{\wt{k}}, s_1, \dots, s_{\wt{l}}]
$$
with the common subset $[0, \infty, t_{\wt{k}}]$.

The rest of the proof of Proposition \ref{St} depends on two statements:
invariance of  integrals under transformations (\ref{trans})
and relation (\ref{relation}). The second one is a form of Arnold's relations by (\ref{calc}).
Invariance of an integral under transformation $\sw{a}$
follows from Proposition \ref{Gst} for 
$$
T_1=[0, \infty, 1, t_a, t_{a+1}, \dots, t_{n}] \quad\mbox{and} \quad T_2=[0, \infty, 1, t_{1}, t_2, \dots, t_a]
$$
with the common subset $[0, \infty, 1, t_a]$
and the involution $\nu$, which interchanges $0$ and $\infty$.
There are three places where invariance
under these transformations is used in the proof
of Proposition \ref{St}: 
(\ref{e2}), (\ref{e6}) and (\ref{e8}).
Corresponding differential forms $\phi$ and $\psi$
from the statement of Proposition \ref{Gst}
may be found from these formulae, 
for (\ref{e2}) and (\ref{e6}) the first form is being of top degree and
for (\ref{e8}) the second form is. 
\end{proof}

\subsection{Formal algebra of periods of $(\Md{*}, \Md{*}\setminus\M{*})$}
If differential forms in Propositions \ref{Gsh} and \ref{Gst}
are regular with logarithmic singularities at infinity,
then forms under integral signs in (\ref{gsh}) and (\ref{gst})
are also regular with logarithmic singularities at infinity.
Thus Propositions \ref{Gsh} and \ref{Gst}
impose quadratic-linear and linear conditions on 
periods of $(\Md{*}, \Md{*}\setminus\M{*})$ or cell-zeta values, which are integrals $\int_{\cyc{S}} \omega$
of a regular  differential form  with logarithmic singularities at infinity
$\omega$ on $\Md{S}$ by the standard simplex. By Proposition 
\ref{Gsh} they form an algebra.

One may consider the formal algebra of periods of $(\Md{*}, \Md{*}\setminus\M{*})$,
which is the one generated by symbols representing integrals as above 
on which all natural relations such as Stokes' theorem,  
and generalized shuffle and generalized stuffle relations are imposed. 
Formal multiple zeta values are elements of this algebra corresponding to 
iterated integrals (\ref{iter}).

The main theorem of \cite{B} states that all
cell-zeta values are rational combinations
of multiple zeta values.
The  long-standing 
conjecture (\cite[Conjecture 1]{IKZ}) states that all rational relations
between multiple zeta values are given by regularized double shuffle relations. 
This leads us to the following conjecture.

\begin{conj}
The formal  algebra of periods of $(\Md{*}, \Md{*}\setminus\M{*})$  is generated by formal 
multiple zeta values and the system of relations to which they obey
is equivalent to  regularized double shuffle relations.
\end{conj}

An analogous conjecture was formulated in \cite{BCS}.
The formal   algebra of cell-zeta values defined there
has the same generators, but relations differ.
Its system of relations contains product map relations, which are the same as
generalized shuffle relations,
dihedral relations and shuffles with respect to one element.
It would be  interesting to compare these algebras.
Note that dihedral relations follow from generalized stuffle relations
for $|T_1|=4$ and $\phi=1$.

\bibliographystyle{alpha}
\bibliography{stuffle}

\begin{thebibliography}{BCS10}

\bibitem[BCS10]{BCS}
Francis Brown, Sarah Carr, and Leila Schneps.
\newblock The algebra of cell-zeta values.
\newblock {\em Compositio Mathematica}, 146(3):731–771, 2010.

\bibitem[Bro09]{B}
Francis C.~S. Brown.
\newblock Multiple zeta values and periods of moduli spaces
  {$\overline{\mathfrak {M}}_{0,n}(\mathbb{R})$}.
\newblock {\em Annales scientifiques de l'\'Ecole Normale Sup\'erieure}, Ser.
  4, 42(3):371--489, 2009.

\bibitem[GM04]{GM}
A.~B. Goncharov and Yu.~I. Manin.
\newblock Multiple {$\zeta$}-motives and moduli spaces
  {$\overline{\mathcal{M}}_{0,n}$}.
\newblock {\em Compositio Mathematica}, 140(1):1–14, 2004.

\bibitem[Gon02]{G}
Alexander~B. Goncharov.
\newblock Periods and mixed motives.
\newblock arXiv:math/0202154, 2002.

\bibitem[IKZ06]{IKZ}
Kentaro Ihara, Masanobu Kaneko, and Don Zagier.
\newblock Derivation and double shuffle relations for multiple zeta values.
\newblock {\em Compositio Mathematica}, 142(2):307–338, 2006.

\bibitem[KY16]{Kaneko}
Masanobu Kaneko and Shuji Yamamoto.
\newblock A new integral–series identity of multiple zeta values and
  regularizations.
\newblock {\em Selecta Mathematica}, 24:2499--2521, 2016.

\bibitem[Mar]{KY}
Nikita Markarian.
\newblock Kaneko--{Y}amamoto integral--series identity and cell-zeta values.
\newblock In preparation.

\bibitem[Rac02]{Racinet}
Georges Racinet.
\newblock Doubles m\'elanges des polylogarithmes multiples aux racines de
  l'unit\'e.
\newblock {\em Publications Math\'ematiques de l'IH\'ES}, 95:185--231, 2002.

\bibitem[Sou10]{S}
Ismael Soud\`eres.
\newblock {M}otivic double shuffle.
\newblock {\em International Journal of Number Theory}, 06(02):339--370, 2010.

\end{thebibliography}

\end{document}